\documentclass[a4paper, 12pt,reqno]{amsart}
\usepackage{amsmath,amssymb,amsthm,enumerate, array}
\usepackage[pagewise]{lineno}
\allowdisplaybreaks
\theoremstyle{plain}
\newtheorem{theorem}{Theorem}
\newtheorem*{theorem*}{Theorem}
\newtheorem{lemma}{Lemma}
\newtheorem*{lemma*}{Lemma}
\theoremstyle{definition}

\newtheorem*{definition*}{Definition}
\theoremstyle{remark}
\newtheorem{remark}{Remark}

\newtheorem*{remark*}{Remark}

\newtheorem*{statement*}{Statement}

\begin{document}
\title[Generalized Salem functions]{Systems of functional equations, the generalized shift, and modelling pathological functions }

\author{Symon Serbenyuk}

\subjclass[2010]{11K55, 11J72, 26A27, 11B34,  39B22, 39B72, 26A30, 11B34.}

\keywords{ Salem function, systems of functional equations,  complicated local structure}

\maketitle
\text{\emph{simon6@ukr.net}}\\
\text{\emph{Kharkiv National University of Internal Affairs, Ukraine }}

\begin{abstract}

The present article is devoted to one  class of generalizations of the Salem functions. To construct such functions by systems of functional equations, the generalized shift operator is used.

\end{abstract}

\section{Introduction}

``In mathematics, when a mathematical phenomenon runs counter to some intuition, then the phenomenon is sometimes called pathological. On the other hand, if a phenomenon does not run counter to intuition, it is sometimes called well-behaved. These terms are sometimes useful in mathematical research and teaching, but there is no strict mathematical definition of pathological or well-behaved." \cite{Wiki}

One can note the following classic examples of a pathology  in real analysis:
\begin{itemize}
\item the Weierstrass function, which  is a continuous everywhere  function but is differentiable nowhere. 

\item The Cantor function, which is a monotonic continuous surjective function that maps $[0,1]$ onto $[0,1]$, but has zero derivative almost everywhere.

\item The Minkowski question-mark function introduced by Minkowski for  illustration of some points about continued fractions; this function  is continuous and strictly increasing but has zero derivative almost everywhere.

\item Fractal sets. The Cantor set defined in terms of the ternary representation of numbers from~$[0,1]$ is the simplest example of  such sets.
\end{itemize}

Fractals in $\mathbb R$ and functions with complicated local structure are the main pathological mathematical objects in real analysis.

``Nowadays, the ternary Cantor set is the paradigmatic model of the fractal geometry \cite{Mandelbrot1999, Taylor2012}  and in many branches of physics (see \cite{Bunde1994}). A large class of Cantor-type sets frequently appear as invariant sets
and attractors of many dynamical systems of the real world problems, see \cite{Kennedy1995}."~(\cite{TSBR2017}). Moran sets and homogeneous Moran sets have very important  applications as fractals. For example, there are applications  in multifractal analysis and  in the study of  the structure of the spectrum of quasicrystals  (see \cite{KLS2016, LW2011, W2005} and references therein), in the power systems (\cite{Feng2005} and its reference)  and measurement of number theory (see \cite{WW2008, Wu2005}), etc. 
For modelling of fractals, one can note that investigations of properties of one analytic rerpresentation of a certain set in terms of various expansions of real numbersare useful because this set can has some different properties in different numeral systems. For example, the set presented by the $P$-representation  is~(\cite{sets}) a self-similar fractal (i.e., this is a Moran set by Moran's definition, \cite{Moran1946}) but  such set defined  in terms of the nega-$P$-representation  is  a non-self-similar set (\cite{sets2}) having the Moran structure (i.e., this is a  Moran set by the definition of Hua et al. (see the  definition in \cite{HRW2000})). Some attention to such topic is given in~\cite{sets1}. In addition, Moran sets play an important role in multifractal analysis/formalism and especially the refined multifractal formalism, and  multifractal formalism ``aims at expressing the dimensions (the Hausdorff and packing dimensions) of the level sets in terms of the Legendre transform of some free energy function in analogy with the usual thermodynamic theory (\cite{{2021-1}, {Selmi2022-ASM1}, {Selmi2022-ASM},  {Yuan2019-N}} and references therein)." (for a full description see in~\cite{sets1}).

 A class of functions with complicated local structure consists of singular (for example, see \cite{D1932,  FSVPDC2012, G1996, {Minkowski}, {Salem1943},  {S.Serbenyuk 2017}, {Zamfirescu1981}}),  nowhere monotonic (\cite{Symon2017, Symon2019}), and nowhere differentiable functions  (for example, see \cite{{Bush1952}, {Serbenyuk-2016}}, etc.).  

An interest in such functions is explained by their connection with modelling  real objects, processes, and phenomena (in physics, economics, technology, etc.), as well as with different areas of mathematics (\cite{ACFS2011, BK2000, Kruppel2009, OSS1995, Sumi2009, Takayasu1984, TAS1993}). For example (\cite{ADF2012}),  in the theory of trigonometric series,   the Riesz products under certain conditions are singular functions. The last functions appear as conjugating homeomorphisms or Perron-Frobenius measures in fractal theory in the context of wavelets as special functions which provides a satisfactory answer to the scale problem under decomposable events in nature such as fingerprints or self-similar behavior in iterated processes.~\cite{ADF2012} To prove the suffciency for the statement on  zero sets of continuous nowhere-differentiable functions, the Takagi non-differentiable  function was applicated by  Lipinski \cite{AK2011}.  The last-mentioned paper  deals with applications of the Takagi function in detail.  A brief historical remark on functions with complicated local structure is also given in~\cite{ACFS2017, {S. Serbenyuk systemy rivnyan 2-2}} (see also references  therein).

Let us consider the Salem function, which is one of the most known examples of singular functions and  is a probability distribution for a random variable defined in terms of  $q$-ary expansions. 

Let $(P_k)$ be a fixed sequence of probability vectors $P_k$ with elements having the following properties for any $k\in\mathbb N$: 
\begin{itemize}
\item $P_k=(p_{0, k}, p_{1, k}, \dots , p_{q-1, k})$;
\item $p_{j, k}>0$ holds for any $j=\overline{0, q-1}$;
\item the equality $p_{0, k}+p_{1, k}+ \dots + p_{q-1, k}=1$  holds;
\item for an arbitrary sequence $(i_k)$, the condition $\prod^{\infty} _{k=1}{p_{i_k, k}}=0$ holds.
\end{itemize}

Let $\eta$ be a random variable, that defined by the following form 
$$
\eta=  \Delta^{q} _{\xi_1\xi_2...\xi_{k}...},
$$
where 
$k=1,2,3, \dots$,   digits $\xi_k$ are  random and taking values $0,1,\dots , q-1$ with probabilities ${p}_{0, k}, {p}_{1, k}, \dots , {p}_{q-1, k}$. That is,  $\xi_n$ are independent and $P\{\xi_k=i_k\}=p_{i_k,k}$, $i_k \in \{0, 1, \dots , q-1\}$. 

\begin{lemma}
The distribution function ${S}_{\eta}$ of the random variable $\eta$ can be
represented by
$$
{S}_{\eta}(x)=\begin{cases}
0,&\text{ $x< 0$;}\\
\beta_{i_1(x), 1}+\sum^{\infty} _{n=2} {\left({\beta}_{i_k(x), k} \prod^{k-1} _{j=1} {{p}_{i_j(x),j}}\right)},&\text{ $0 \le x<1$;}\\
1,&\text{ $x\ge 1$.}
\end{cases}
$$
\end{lemma}
\begin{proof} This proof is based on the definition of a probability distribution. Really, let $k\in\mathbb N$; then our statement follows from the equalities 
$$
\{\eta<x\}=\{\xi_1<i_1(x)\}\cup\{\xi_1=i_1(x),\xi_2<i_2(x)\}\cup \ldots 
$$
$$
\ldots \cup\{\xi_1=i_1(x),\xi_2=i_2(x),\dots , \xi_{k-1}=i_{k-1}(x),\xi_{k}<i_{k}(x)\}\cup \ldots ,
$$
$$
P\{\xi_1=i_1(x),\xi_2=i_2(x),\dots , \xi_{k-1}=i_{k-1}(x), \xi_{k}<i_{k}(x)\}=\beta_{i_{k}(x),k}\prod^{k-1} _{j=1} {{p}_{i_{j}(x),j}}.
$$
\end{proof}

The last distribution function is the classical Salem function introduced in \cite{Salem1943} whenever the condition $P_k=P=(p_0, p_1, \dots , p_{q-1})$ holds for all $k\in\mathbb N$. That is, 
\begin{equation}
\label{eq: f2}
S(x)=\beta_{i_1}+ \sum^{\infty} _{k=2} {\left(\beta_{i_k}\prod^{k-1} _{i=1}{p_i}\right)}:=\Delta^{P} _{i_1i_2...i_k...}=s\in [0,1],
\end{equation}
where $q>1$ is a fixed positive integer,
$$
\sum^{\infty} _{k=1}{\frac{i_k}{q^k}}:=\Delta^q _{i_1i_2...i_k...}=x\in [0,1]
$$
 is the $q$-ary representation of a number $x$, and $i_k \in \{0, 1, \dots , q-1\}$. Also,
$$
\beta_{i_k}=\begin{cases}
0&\text{if $i_{k}=0$}\\
\sum^{i_{k}-1} _{l=0} {p_l}&\text{if $i_{k}\ne 0$.}
\end{cases}
$$

It is easy to see that the Salem function is a $q$-ary expansion whenever the condition $p_0=p_1=\dots = p_{q-1}=\frac 1 q$ holds.

 Many researches were devoted to the Salem function and its generalizations (for example, see \cite{ACFS2017, Kawamura2010, Symon2015, Symon2017, Symon2019} and references in these papers). Since generalizations of the Salem function can be non-differentiable functions or those that do not have a derivative on a certain set, it is useful to apply various expansions of real numbers for modelling these functions and for studying images of fractals (for example, corresponding explanations are given in \cite{sets1, sets2, 11}).  Using auxiliary maps and systems of functional equations is  an useful  technique for constructing generalizations of the Salem functions  (for example, see \cite{{Symon2024}, 3}). Also, one can note that techniques used under molelling generalizations of the Salem function in terms of alternating expansion, are more complicated than in the case of corresponding positive expansions (\cite{{Symon2021}, 2, 4}). In addition, since  many researches devoted to constructing,  investigations, and applications of various numeral systems (for example, see~\cite{{Renyi1957}, SLS2025, {S.Serbenyuk},  {preprint1-2018}, {Symon2023}} and references therein), the consideration of properties of functions (including generalizations of the Salem functions) and sets in various numeral systems is an useful tool to improve techniques of pure and applied mathematics in the topic of  mathematical pathologies.

 This research is a generalization of  \cite{Symon2021} (in which considered the case of $q$-ary expansions of arguments) and is devoted to a certain generalized Salem function, which  modificated by  the generalized shift operator, with arguments defined by the Salem function.
That is, let us consider properties of a function of the form
$$
G(s)=\gamma_{i_{n_1}}+\sum^{\infty} _{k=2}{\left(\gamma_{i_{n_k}}\prod^{k-1} _{t=1}{r_{i_{n_t}}}\right)},
$$
where
$$
\gamma_{i_{n_k}}=\begin{cases}
0&\text{if ${i_{n_k}}=0$}\\
\sum^{{i_{n_k}}-1} _{l=0} {r_l}&\text{if ${i_{n_k}}\ne 0$.}
\end{cases}
$$
Here $q>1$ is a fixed positive integer, $s$ is a value of the Salem function $S$, and $R=(r_0, r_1, \dots , r_{q-1})$  is a probability vector with  $r_j>0$ for all $j=\overline{0, q-1}$.   Also, 
 $(n_k)$ is a fixed sequence of positive integers such that the following conditions hold: $n_i\ne n_j$ for $i\ne j$;  for any $n\in\mathbb N$, there exists a number $k_0$ such that $n_{k_0}=n$.

\section{Auxiliary notions}

In this section, auxiliary definitions and explanations are considered. Let us begin with some connections with the geometry of numbers.

\subsection{Numbers having two expansions and cylinders.}

Let us consider a  representation  by the function $S_\eta$ for  values from $[0,1]$, i.e.,
\begin{equation}
\label{eq: f1}
[0,1]\ni s=\Delta^{(P_k)} _{i_1i_2...i_k...}:=\beta_{i_1, 1}+\sum^{\infty} _{k=2} {\left({\beta}_{i_k, k} \prod^{k-1} _{j=1} {{p}_{i_j,j}}\right)}= S_{\eta}(x).
\end{equation}

Let us note that certain numbers from $[0,1]$ have two different representations by expansion~\eqref{eq: f1}, i.e., 
$$
\Delta^{(P_k)} _{i_1i_2\ldots i_{m-1}i_m000\ldots}=\Delta^{(P_k)} _{i_1i_2\ldots i_{m-1}[i_m-1][q-1][q-1]\ldots}=\beta_{i_1, 1}+\sum^{m} _{k=2} {\left({\beta}_{i_k, k} \prod^{k-1} _{j=1} {{p}_{i_j,j}}\right)}.
$$
Such numbers are called \emph{$(P_k)$-rational}. The other numbers in $[0,1]$ are called \emph{$(P_k)$-irrational}.

By analogy, numbers of the form 
$$
\Delta^{P} _{i_1i_2\ldots i_{m-1}i_m000\ldots}=\Delta^{P} _{i_1i_2\ldots i_{m-1}[i_m-1][q-1][q-1]\ldots}=\beta_{i_1}+\sum^{m} _{k=2} {\left({\beta}_{i_k} \prod^{k-1} _{j=1} {{p}_{i_j}}\right)}
$$
are \emph{$P$-rational}. The other numbers in $[0,1]$ are \emph{$P$-irrational}.

Let $c_1,c_2,\dots, c_m$ be a fixed 
ordered tuple of integers such that $c_j\in\{0,1,\dots, q-~1\}$ for $j=\overline{1,m}$. 

\emph{A cylinder $\Lambda^{(P_k)} _{c_1c_2...c_m}$ of rank $m$ with base $c_1c_2\ldots c_m$} is the following set 
$$
\Lambda^{(P_k)} _{c_1c_2...c_m}\equiv\{x: x=\Delta^{(P_k)} _{c_1c_2...c_m i_{m+1}i_{m+2}\ldots i_{m+k}\ldots}\}.
$$
That is,  any cylinder $\Lambda^{(P_k)} _{c_1c_2...c_m}$ is a closed interval of the form
$$
\left[\Delta^{(P_k)} _{c_1c_2...c_m000...}, \Delta^{(P_k)} _{c_1c_2...c_m[q-1][q-1][q-1]...}\right].
$$

By analogy, in the case of representation \eqref{eq: f2}, we get that  an arbitrary cylinder $\Lambda^{P} _{c_1c_2...c_m}$ is a closed interval of the form
$$
\left[\Delta^P _{c_1c_2...c_m000...}, \Delta^P _{c_1c_2...c_m[q-1][q-1][q-1]...}\right].
$$

\subsection{Shifts of digits}

Let us consider the shift  and generalized shift operators  for the cases of expansions \eqref{eq: f2} and  \eqref{eq: f1}.

 \emph{The shift operator $\sigma$ of expansion \eqref{eq:  f1}} is of the following form
$$
\sigma(s)=\sigma\left(\Delta^{(P_k)} _{i_1i_2\ldots i_k\ldots}\right)=\beta_{i_2, 2}+\sum^{\infty} _{k=3} {\left({\beta}_{i_k, k} \prod^{k-1} _{j=1} {{p}_{i_j,j}}\right)}=\frac{1}{p_{0, 1}}\Delta^{(P_k)} _{0i_2\ldots i_k\ldots}.
$$
It is easy to see that 
\begin{equation*}
\begin{split}
\sigma^m(s) &=\sigma^m\left(\Delta^{(P_k)} _{i_1i_2\ldots i_k\ldots}\right)\\
& =\beta_{i_{m+1}, m+1}+\sum^{\infty} _{k=m+2} {\left({\beta}_{i_k, k} \prod^{k-1} _{j=1} {{p}_{i_j,j}}\right)}=\Delta^{(P_k)} _{\underbrace{0\ldots 0}_{m}i_{m+1}i_{m+2}\ldots}\prod^{m} _{l=1}{\frac{1}{p_{0, l}}}.
\end{split}
\end{equation*}
Therefore, 
\begin{equation}
\label{eq: operator3}
x=\beta_{i_{1}, 1}+\sum^{m} _{k=2} {\left({\beta}_{i_k, k} \prod^{k-1} _{j=1} {{p}_{i_j,j}}\right)}+\sigma^n(x) \prod^{m} _{t=1}{p_{i_t, t}}.
\end{equation}
By analogy,
$$
\sigma^m\left(\Delta^P _{i_1i_2\ldots i_k\ldots}\right) =\beta_{i_{m+1}}+\sum^{\infty} _{k=m+2} {\left({\beta}_{i_k} \prod^{k-1} _{j=1} {{p}_{i_j}}\right)}=\frac{1}{p^m _0}\Delta^{P} _{\underbrace{0\ldots 0}_{m}i_{m+1}i_{m+2}\ldots}=\Delta^{P} _{i_{m+1}i_{m+2}\ldots}.
$$

Let us consider the generalized shift operator (see \cite{Symon2021} and references therein) for the case of representation \eqref{eq: f1}.

This  idea includes the following:  any number from a certain closed  interval can be represented by two fixed sequences $(P_k)$  and  $(i_k)$. The generalized shift operator maps the preimage into a number represented by using the following two sequences $(P_1, P_2,\dots , P_{m-1}, P_{m+1}, P_{m+2}, \dots )$ and $(i_1, i_2,\dots , i_{m-1}, i_{m+1}, i_{m+2}, \dots )$. 

Suppose a number $x\in [0,1]$ is  represented by representation \eqref{eq: f1}. Then
$$
s=v_{m-1}+\beta_{i_m, m}\prod^{m-1} _{k=1}{p_{i_k, k}}+p_{i_m,m}w_{m+1},
$$
where 
$$
v_{m-1}=\beta_{i_1, 1}+\sum^{m-1} _{k=2} {\left({\beta}_{i_k, k} \prod^{k-1} _{j=1} {{p}_{i_j,j}}\right)}=\Delta^{(P_k)} _{i_1i_2\ldots i_{m-1}000...}
$$
and
$$
w_{m+1}=\left(\beta_{i_{m+1}, m+1}+\sum^{\infty} _{t=m+2} {\left({\beta}_{i_t, t} \prod^{t-1} _{u=1} {{p}_{i_u, u}}\right)}\right)\prod^{m-1} _{j=1}{p_{i_j, j}}=\sigma^m (s)\prod^{m-1} _{j=1}{p_{i_j, j}}.
$$
Hence,
\begin{equation}
\label{eq: generalized shift 1}
\sigma_m(s)=\frac{x- (1-p_{i_m, m})v_{m-1}-\beta_{i_m, m}\prod^{m-1} _{k=1}{p_{i_k, k}}}{p_{i_m, m}}.
\end{equation}
One can remark that $\sigma(s)=\sigma_1(s)$.

By analogy with explanations in \cite{Symon2021}, let  us consider some remarks on compositions of shift operators in the case of our representations of numbers. 

\begin{remark}
Suppose $x=\Delta^{(P_k)} _{i_1i_2...i_k...}$ and $m$ is a fixed positive integer. Then
$$
\sigma_{m}(s)=\sigma_{m}\left(\Delta^{(P_k)} _{i_1i_2...i_k...}\right)=\Delta^{(P_k)\setminus \{P_m\}} _{i_1i_2...i_{m-1}i_{m+1}...},
$$
\begin{align*}
\sigma_{m}\circ \sigma_{m}(s)&=\sigma^{2} _{m}(s)=\sigma^{2} _{m}\left(\Delta^{(P_k)} _{i_1i_2...i_k...}\right)\\
&=\sigma_m\left(\sigma_m(\Delta^{(P_k)} _{i_1i_2...i_k...})\right)=\sigma_{m}\left(\Delta^{(P_k)\setminus \{P_m\}} _{i_1i_2...i_{m-1}i_{m+1}...}\right)=\Delta^{(P_k)\setminus \{P_m, P_{m+1}\}} _{i_1i_2...i_{m-1}i_{m+2}...},
\end{align*}
and
$$
\underbrace{\sigma_{m}\circ \ldots \circ \sigma_{m}(s)}_{n}=\sigma^{n} _{m}\left(\Delta^{(P_k)} _{i_1i_2...i_k...}\right)=\Delta^{(P_k)\setminus \{P_m, P_{m+1}, \dots , P_{m+n-1}\}} _{i_1i_2...i_{m-1}i_{m+n}i_{m+n+1}...}.
$$

For the case of the shift operator, we get
\begin{equation*}
\begin{split}
\sigma^n(s) &=\sigma^n\left(\Delta^{(P_k)} _{i_1i_2\ldots i_k\ldots}\right)=\Delta^{(P_k)\setminus \{P_1, P_{2}, \dots , P_{n}\}} _{i_{n+1}i_{n+2}..i_{n+k}...}\
\end{split}
\end{equation*}
\end{remark}
\begin{remark} (By analogy with \cite{Symon2021}).
Using the last remark, now let us describe a more general case. Suppose that $n_1$ and $n_2$ are two positive integers. Then
\begin{equation}
\label{eq: 2-composition}
\sigma_{n_2}\circ \sigma_{n_1}(s)=\sigma_{n_2}\left(\Delta^{(P_k)\setminus \{P_{n_1}\}} _{i_1i_2...i_{n_1-1}i_{n_1+1}...}\right)=\begin{cases}
\Delta^{(P_k)\setminus \{P_{n_1}, P_{n_2}\}} _{i_1i_2...i_{n_2-1}i_{n_2+1}...i_{n_1-1}i_{n_1+1}...}&\text{if $n_1>n_2$}\\
\Delta^{(P_k)\setminus \{P_{n_1}, P_{n_2+1}\}} _{i_1i_2...i_{n_1-1}i_{n_1+1}...i_{n_2-1}i_{n_2}i_{n_2+2}...}&\text{if $n_1<n_2$}\\
\Delta^{(P_k)\setminus \{P_{n_0}, P_{n_0+1}\}} _{i_1i_2...i_{n_0-1}i_{n_0+2}...}&\text{if $n_1=n_2=n_0$.}
\end{cases}
\end{equation}
That is, the last rule is of the following form:

\begin{center}
\begin{tabular}{|c|c| c| c|}
\hline
The case for $n_2$ & $ n_1=n_2=n_0$ & $ n_2> n_1$ & $ n_2 < n_1$\\
\hline 
What will a digit be deleted? & $i_{n_0+1}$ & $i_{n_2+1}$ & $i_{n_2}$  \\
\hline
\end{tabular}
\end{center}

Let us remark that, in the case  of expansions~\eqref{eq: f2}, we have
\begin{equation}
\label{eq: generalized shift 2}
\sigma_m(\Delta^P _{i_1i_2...i_{k-1}i_k...})=\frac{x- (1-p_{i_m})v_{m-1}-\beta_{i_m}\prod^{m-1} _{k=1}{p_{i_k}}}{p_{i_m}}=\Delta^P _{i_1i_2...i_{m-1}i_{m+1}...},
\end{equation}
where $v_{m-1}=\Delta^P _{i_1i_2...i_{m-1}000...}$.
\end{remark}

Let us note some auxiliary property which is useful for modelling the main object of this research. 
\begin{remark} (An analogy with \cite{Symon2021}).
\label{rm: the main remark}
Suppose that numbers $s\in [0,1]$ are represented in terms of  representation~\eqref{eq: f2} and we need to delete the digits $i_{n_1},i_{n_2}, \dots , i_{n_k}$ (according to this fixed order) by using a composition of the generalized shift operators in $s=\Delta^P _{i_1i_2...i_k...}$. Here $(n_k)$ is a finite fixed sequence of positive integers such that $n_i\ne n_j$ for $i\ne j$. That is, we model
$$
s_0=\Delta^P _{i_1i_2...i_{n_1-1}i_{n_1+1}...i_{n_2-1}i_{n_2+1}...i_{n_k-1}i_{n_k+1}i_{n_k+2}...i_{n_k+t}...},~~~\mbox{where}~t=1,2,3, \dots .
$$

Using \eqref{eq: 2-composition}, we obtain a certain sequence $(\bar n_k)$ of positive integers. That is, for all $i=\overline{1,k}$, 
$$
\bar n_i= n_i-\varrho_i,
$$
where $\varrho_i$ is the number of all numbers which are less than $n_i$ in the finite sequence $n_1, n_2, \dots , n_i$.

So,
$$
s_0=\sigma_{\bar n_k}\circ \sigma_{\bar n_{k-1}} \circ \dots \circ \sigma_{\bar n_2} \circ \sigma_{\bar n_1}(s).
$$
\end{remark}

\begin{lemma}
In the case of expansion \eqref{eq: f2}, the generalized shift operator has the following properties:
\begin{itemize}
\item The mapping $\sigma_m$ is continuous at each point of the interval $\left(\inf\Delta^P _{c_1c_2...c_m}, \sup\Delta^P _{c_1c_2...c_m}\right)$. The endpoints of $\Delta^P _{c_1c_2...c_m}$ are  points of discontinuity of the mapping.
\item The mapping $\sigma_m$ has a derivative almost everywhere (with respect to
the Lebesgue measure). If the mapping  has a derivative at the point $s=\Delta^P _{i_1i_2...i_k...}$, then $\left(\sigma_m\right)^{'}=\frac{1}{p_{i_m}}$.
\end{itemize}
\end{lemma}
 
  All  properties follow from the definition of  $\sigma_m$  and equality \eqref{eq: generalized shift 1}.  It is easy to see that $\sigma_m$ is a piecewise linear function.


\section{The main results}

Suppose $q>1$ is a fixed positive integer, $P=(p_0, p_1, \dots , p_{q-1})$, and  $R=(r_0, r_1, \dots , r_{q-1})$  are finite fixed sequences of numbers such that the conditions  $p_j>0, r_j >0$, $j=\overline{0, q-1}$,  as well as $p_0+p_1+\dots + p_{q-1}=1$ and $r_0+r_1+\dots + r_{q-1}=1$ hold.

Suppose $(n_k)$ is a fixed sequence of positive integers such that $n_i\ne n_j$ for $i\ne j$ and such that  for any $n\in\mathbb N$ there exists a number $k_0$ for which the condition $n_{k_0}=n$ holds. Assume that $\bar n_k=n_k-\varrho_k$ for all $k=1,2, 3, \dots$, where $\varrho_k$ is the number of all numbers which are less than $n_k$ in the finite sequence $n_1, n_2, \dots , n_k$.

Suppose
$$
\beta_{i_1}+ \sum^{\infty} _{k=2} {\left(\beta_{i_k}\prod^{k-1} _{i=1}{p_i}\right)}=\Delta^{P} _{i_1i_2...i_k...}=s\in [0,1],
$$
$$
\beta_{i_{k}}=\begin{cases}
0&\text{if ${i_{k}}=0$}\\
\sum^{{i_{k}}-1} _{l=0} {p_l}&\text{if ${i_{k}}\ne 0$,}
\end{cases}
$$
and
$$
\gamma_{i_{n_k}}=\begin{cases}
0&\text{if ${i_{n_k}}=0$}\\
\sum^{{i_{n_k}}-1} _{l=0} {r_l}&\text{if ${i_{n_k}}\ne 0$.}
\end{cases}
$$

\begin{theorem}
Let $R=\{r_0,r_1,\dots , r_{q-1}\}$ be a fixed tuple of real numbers such that $r_j\in (-1,1)$, where $j=\overline{0,q-1}$, $\sum_j{p_j}=1$, and $0=\gamma_0<\gamma_j=\sum^{j-1} _{j=0}{p_j}<1$ for all $j\ne 0$. Then the following system of functional equations
\begin{equation}
\label{eq: system-q}
f\left(\sigma_{\bar n_{k-1}}\circ \sigma_{\bar n_{k-2}}\circ \ldots \circ \sigma_{\bar n_1}(s)\right)=\gamma_{i_{n_k}}+r_{i_{n_k},}f\left(\sigma_{\bar n_{k}}\circ \sigma_{\bar n_{k-1}}\circ \ldots \circ \sigma_{\bar n_1}(s)\right),
\end{equation}
where $s=\Delta^P _{i_1i_2...i_k...}$, $k=1,2, \dots$, and $\sigma_0(s)=s$, has the unique solution
$$
G(s)=\gamma_{i_{n_1}}+\sum^{\infty} _{k=2}{\left(\gamma_{i_{n_k}}\prod^{k-1} _{t=1}{r_{i_{n_t}}}\right)}
$$
in the class of determined and bounded on $[0, 1]$ functions. 
\end{theorem}
\begin{proof}
The function $G$ is a determined on $[0,1]$ function. Hence we obtain the following using system~\eqref{eq: system-q}: 
\begin{align*}
G(s) &= \gamma_{i_{n_1}}+r_{i_{n_1}}f(\sigma_{\bar n_1}(s))\\
&=\gamma_{i_{n_1}}+r_{i_{n_1}}(\gamma_{i_{n_2}}+r_{i_{n_2}}f(\sigma_{\bar n_2}\circ\sigma_{\bar n_1}(s)))=\dots \\
\dots &=\gamma_{i_{n_1}}+\gamma_{i_{n_2}}r_{i_{n_1}}+\gamma_{i_{n_3}}r_{i_{n_1}}r_{i_{n_2}}+\dots +\gamma_{i_{n_k}}\prod^{k-1} _{j=1}{r_{i_{n_j}}}+\left(\prod^{k} _{t=1}{r_{i_{n_t}}}\right)f(\sigma_{\bar n_k}\circ \dots \circ \sigma_{\bar n_2}\circ \sigma_{\bar n_1}(s)).
\end{align*}
Whence,
$$
G(s)=\gamma_{i_{n_1}}+\sum^{\infty} _{k=2}{\left(\gamma_{i_{n_k}}\prod^{k-1} _{j=1}{r_{i_{n_j}}}\right)},
$$
since $G, f$ are  determined and bounded on $[0,1]$ functions and 
$$
\lim_{k\to\infty}{f(\sigma_{\bar n_k}\circ \dots \circ \sigma_{\bar n_2}\circ \sigma_{\bar n_1}(s))\prod^{k} _{t=1}{r_{i_{n_t}}}}=0,
$$
where
$$
\prod^{k} _{t=1}{r_{i_{n_t}}}\le \left( \max_{0\le u\le q-1}{r_u}\right)^k\to 0, ~~~ k\to \infty.
$$
\end{proof}

\begin{theorem} The following properties hold: 
\begin{itemize}
\item The function $G$ is continuous at any $P$-irrational point of $[0,1]$.
\item The function $G$ is continuous at the $P$-rational point
$$
s_0=\Delta^{P} _{i_1i_2...i_{m-1}i_m 000...}=\Delta^{P} _{i_1i_2...i_{m-1}[i_m-1] [q-1][q-1][q-1]...}
$$
whenever a sequence $(n_k)$ is such that the conditions  $k_0=\max\{k:  n_k \in \{1,2,\dots, m\}\}$, $n_{k_0}=m$, and $n_1, n_2, \dots , n_{k_0-1}\in \{1, 2, \dots , m-1\}$  hold. Otherwise, the $P$-rational  point  $s_0$ is a point of discontinuity.
\item The set $G_D$ of all points of discontinuity of the function $G$ is a countable, finite, or empty set. It  depends on the sequence $(n_k)$.
\end{itemize}
\end{theorem}
\begin{proof}
Let us note that a certain fixed function $G$ is given by a fixed sequence $(n_k)$ described above. One can write our  mapping by the following:
$$
G: s=\Delta^P _{i_1i_2...i_k...}\to ~\gamma_{i_{n_1}}+\sum^{\infty} _{k=2}{\left(\gamma_{i_{n_k}}\prod^{k-1} _{l=1}{r_{i_{n_l}}}\right)}=\Delta^{G(s)} _{i_{n_1}i_{n_2}...i_{n_k}...}=G(s)=y.
$$

Let $s_0=\Delta^P _{i_1i_2...i_k...}$ be an arbitrary $P$-irrational number from $[0,1]$. Let $s=\Delta^P _{\theta_1\theta_2...\theta_k...}$ be a  $P$-irrational number such that the condition $\theta_{n_j}=i_{n_j}$ holds  for all $j=\overline{1,k_0}$, where $k_0$ is a certain positive integer. 
Then
$$
G(x_0)=\Delta^{G(s)} _{i_{n_1}i_{n_2}...i_{n_{k_0}}i_{n_{k_0+1}}...},
$$
$$
G(x)=\Delta^{G(s)} _{i_{n_1}i_{n_2}...i_{n_{k_0}}\theta_{n_{k_0+1}}...\theta_{n_{k_0}+k}...}.
$$
Since $G$ is a bounded function, $ G(s) \le 1$, we have 
\begin{align*}
&G(s)-G(s_0) = \\
&=\left(\prod^{k_0} _{j=1}{r_{i_{n_j}}}\right) \left(\gamma_{\theta_{n_{k_0+1}}}+\sum^{\infty} _{t=2}{\left(\gamma_{\theta_{n_{k_0+t}}}\prod^{k_0+t-1} _{u=k_0+1}{r_{\theta_{n_u}}}\right)}-\gamma_{i_{n_{k_0+1}}}-\sum^{\infty} _{t=2}{\left(\gamma_{i_{n_{k_0+t}}}\prod^{k_0+t-1} _{u=k_0+1}{r_{i_{n_u}}}\right)}\right) \\
&=\left(\prod^{k_0} _{j=1}{r_{i_{n_j}}}\right)\left(G(\sigma_{\bar n_{k_0}}\circ\ldots \sigma_{\bar n_2} \circ \sigma_{\bar n_1}(s))-G(\sigma_{\bar n_{k_0}}\circ\ldots \sigma_{\bar n_2} \circ \sigma_{\bar n_1}(s_0))\right),
\end{align*}
and
$$
|G(s)-G(s_0)|\le \delta\prod^{k_0} _{j=1}{r_{i_{n_j}}}\le \delta\left(\max\{r_0,\dots , r_{q-1}\}\right)^{k_0}\to 0 ~~~~~~~(k_0\to\infty). 
$$
Here $\delta$ is a certain real number.

So, $\lim_{s\to s_0}{G(x)}=G(s_0)$, i.e., the function $G$ is continuous at any $P$-irrational point. 

Let $s_0$ be a $P$-rational number, i.e.,
$$
s_0=s^{(1)} _0=\Delta^{P} _{i_1i_2...i_{m-1}i_m 000...}=\Delta^{P} _{i_1i_2...i_{m-1}[i_m-1] [q-1][q-1][q-1]...}=s^{(2)} _0.
$$
Then there exist positive integers $k^{*}$ and $k_0$ such that
$$
y_1=G\left(s^{(1)} _0\right)=\Delta^{G(s)} _{i_{n_1}i_{n_2}...i_{n_{k^{*}}}...i_{n_{k_0}}000...},
$$
$$
y_2=G\left(s^{(2)} _0\right)=\Delta^{G(s)} _{i_{n_1}i_{n_2}...i_{n_{k^{*}-1}}[i_{n_{k^{*}}}-1]i_{n_{k^{*}+1}}...i_{n_{k_0}}[q-1][q-1][q-1]...}.
$$
 Here $n_{k^{*}}=m$, $n_{k^{*}}\le n_{k_0}$, as well as $k_0$ is a number such that $i_{n_{k_0}}\in\{i_1, \dots, i_{m-1}, i_m\}$ and ${k_0}$ is the maximum position of any number from  $\{1,2,\dots , m\}$ in the sequence $(n_k)$.

Using the case of a $P$-irrational number, let us consider the limits
$$
\lim_{s\to s_0+0}{G(s)}=\lim_{s\to s^{(1)} _0}{G(s)}=G(x^{(1)} _0)=y_1,~~~\lim_{s\to s_0-0}{G(s)}=\lim_{s\to s^{(2)} _0}{G(s)}=G(s^{(2)} _0)=y_2.
$$

Whence $y_1=y_2$ whenever a sequence $(n_k)$ is such that the conditions $n_{k_0}=m$, $k_0=\max\{k:  n_k \in \{1,2,\dots, m\}\}$,  and $n_1, n_2, \dots , n_{k_0-1}\in \{1, 2, \dots , m-1\}$ hold.

So,  the set $G_D$ of all points of discontinuity of the function $G$ is a countable, finite, or empty set. It  depends on the sequence $(n_k)$.
\end{proof}

\begin{remark}
From the last theorem it is follows the following relations between a sequence $(n_k)$, $k=1, 2, 3, \dots$, and  the set $G_D$ of all points of discontinuity of the function $G$.

\begin{center}
\begin{tabular}{|c|c| c| c|}
\hline
Relations & $(n_k)=(k)$ &   $n_k\ne k$ for a finite number of $k$  & the other case\\
\hline 
 The set  $G_D$ is ...& empty & finite & countable   \\
\hline
\end{tabular}
\end{center}
\end{remark}

Let us consider the monotonicity and differential properties.

Suppose $(n_k)$ is a fixed sequence and $c_{n_1}, c_{n_2}, \dots , c_{n_t}$ is a fixed tuple of numbers $c_{n_j}\in\{0,1,\dots , q-1\}$, where $j=\overline{1,t}$ and $t$ is a fixed positive integer.

We begin with the following set
$$
\mathbb S_{P, (c_{n_t})}\equiv \left\{x: x=\Delta^P _{i_1i_2...i_{n_1-1}\overline{c_{n_1}}i_{n_1+1}...i_{n_2-1}\overline{c_{n_2}}...i_{n_{t}-1}\overline{c_{n_t}}i_{n_t+1}...i_{n_t+k}...}\right\},
$$
where $k=1,2,\dots $, and $\overline{c_{n_j}}\in \{c_{n_1}, c_{n_2}, \dots , c_{n_t}\}$ for all $j=\overline{1,t}$. This set has non-zero Lebesgue measure (for example, similar sets are investigated in  terms of other representations of numbers in~\cite{S. Serbenyuk alternating Cantor series 2013}). It is easy to see that 
$\mathbb S_{P, (c_{n_t})}$ maps to
$$
G\left(\mathbb S_{P, (c_{n_t})}\right)\equiv\left\{y: y=\Delta^{G(s)} _{c_{n_1} c_{n_2}\dots  c_{n_t}i_{n_{t+1}}...i_{n_{t+k}}...}\right\}
$$
under $G$.

For a  value $\mu_G \left(\mathbb S_{P, (c_{n_t})}\right)$ of the increment, the following is true.
$$
\mu_G \left(\mathbb S_{P, (c_{n_t})}\right)=G\left(\sup\mathbb S_{P, (c_{n_t})}\right)-G\left(\inf\mathbb S_{P, (c_{n_t})}\right)=\Delta^{G(s)} _{c_{n_1} c_{n_2}\dots  c_{n_t}[q-1][q-1]...}-\Delta^{G(s)} _{c_{n_1} c_{n_2}\dots  c_{n_t}00...}=\prod^{t} _{j=1}{r_{c_{n_j}}}.
$$
Let us  consider the intervals
$\left[\inf\mathbb S_{P, (c_{n_t})}, \sup\mathbb S_{P, (c_{n_t})}\right]$. That is,  
\begin{align*}
&\sup\mathbb S_{P, (c_{n_t})}-\inf\mathbb S_{P, (c_{n_t})}= \\
&=\Delta^{P} _{\underbrace{[q-1][q-1]...[q-1]}_{n_1-1}\overline{c_{n_1}}\underbrace{[q-1][q-1]...[q-1]}_{n_2-1}\overline{c_{n_2}}...\underbrace{[q-1][q-1]...[q-1]}_{c_t-1}\overline{c_{n_t}}[q-1][q-1][q-1]...}\\
&-\Delta^{P} _{\underbrace{00...0}_{n_1-1}\overline{c_{n_11}}\underbrace{00...0}_{n_2-1}\overline{c_{n_2}}...\underbrace{00...0}_{c_t-1}\overline{c_{n_t}}000...}
\end{align*}

Whence,
$$
0<\sup\mathbb S_{P, (c_{n_t})}-\inf\mathbb S_{P, (c_{n_t})}<1
$$
and 
\begin{equation}
\label{eq: increment}
\mu_G \left(\mathbb S_{P, (c_{n_t})}\right)=\mu_G \left(\left[\inf\mathbb S_{P, (c_{n_t})}, \sup\mathbb S_{P, (c_{n_t})}\right]\right)=\prod^{t} _{j=1}{r_{c_{n_j}}}.
\end{equation}

One can note that the function $G$ is a monotonic increasing generalization of the Salem function defined in terms of the $P$-representation whenever $(n_k)=(k)$ (i.e., $(\bar n_k)=const=1$) and all $r_j (j=\overline{0,q-1})$ are positive or non-negative numbers.

So, one can formulate the following statements.
\begin{theorem}
The function $G$ has the following properties:
\begin{enumerate}
\item If $r_j\ge 0$ or $r_j>0$ for all $j=\overline{0,q-1}$, then:
\begin{itemize}
\item $G$ does not have intervals of monotonicity on $[0,1]$ whenever the condition $n_k=k$ holds for no  more than a finite number of values of $k$; 
\item $G$ has at least one interval of monotonicity on $[0,1]$ whenever  the condition $n_k\ne k$ holds for  a finite number of values of $k$; 
\item $G$ is a monotonic non-decreasing function (in the case when $r_j\ge 0$ for all $j=\overline{0,q-1}$) or is a strictly increasing function (in the case when $r_j> 0$ for all $j=\overline{0,q-1}$) whenever the condition $n_k=k$ holds for  $k\in\mathbb N$.
\end{itemize}
\item If there exists   $r_j=0$, where $j=\overline{0,q-1}$, then $G$ is  constant almost everywhere on $[0,1]$.
\item If there exists  $r_j<0$ (other $r_j$ are positive), where $j=\overline{0,q-1}$, and the condition $n_k=k$ holds for  almost all $k\in\mathbb N$, then $G$ does not have intervals of monotonicity on $[0,1]$.
\end{enumerate}
\end{theorem}

Let us note that  the last statements follow from~\eqref{eq: increment}.

Let us consider a cylinder $\Delta^P _{c_1c_2...c_n}$. We obtain 
$\mu_G\left(\Delta^P _{c_1c_2...c_n}\right)=$
$$
=\Delta^{G(s)} _{\underbrace{[q-1][q-1]...[q-1]}_{e_1}\overline{c_{1}}\underbrace{[q-1][q-1]...[q-1]}_{e_2}\overline{c_{2}}...\underbrace{[q-1][q-1]...[q-1]}_{e_n}\overline{c_{n}}[q-1][q-1][q-1]...}
$$
$$
-\Delta^{G(s)} _{\underbrace{00...0}_{e_1}\overline{c_{1}}\underbrace{00...0}_{e_2}\overline{c_{2}}...\underbrace{00...0}_{e_n}\overline{c_{n}}000...},
$$
where $\overline{c_{j}}\in\{c_1,c_2,\dots , c_n\}$, $j=\overline{1,n}$, and $(e_n)$ is a certain sequence of numbers from $\mathbb N \cup\{0\}$.

So, differential properties of $G$ depend on a sequence $(n_k)$ and the set of numbers $R=\{r_0, r_1,\dots , r_{q-1}\}$.
\begin{statement*}
The function $G$ can be a singular or non-differentiable function. It  depends on the sequence $(n_k)$ and $R=(r_0,r_1,\dots , r_{q-1})$.
\end{statement*}

In addition, let us note the following.
\begin{lemma}
Let $\eta$ be a random variable  defined by the following form 
$$
\eta=  \Delta^{P} _{\xi_{1}\xi_{2}...\xi_{k}...},
$$
where
$\xi_k=i_{n_k}$, $k=1,2,3,\dots $, and the digits $\xi_{k}$ are  random and take the values $0,1,\dots ,q-1$ with probabilities ${r}_{0}, {r}_{1}, \dots , {r}_{q-1}$.
That is,  $\xi_n$ are independent and $P\{\xi_{k}=i_{n_k}\}=r_{i_{n_k}}$, $i_{n_k}\in\{0,1,\dots q-1\}$. 
 Here $(n_k)$ is a sequence of positive integers such that $n_i\ne n_j$ for $i\ne j$ and such that  for any $n\in\mathbb N$ there exists a number $k_0$ for which the condition $n_{k_0}=n$ holds.

The distribution function ${F}_{\eta}$ of the random variable $\eta$ can be
represented by
$$
{F}_{\eta}(s)=\begin{cases}
0,&\text{ $s< 0$}\\
\gamma_{i_{n_1}(s)}+\sum^{\infty} _{k=2} {\left({\gamma}_{i_{n_k}(s)} \prod^{k-1} _{u=1} {{r}_{i_{n_u}(s)}}\right)},&\text{ $0 \le s<1$}\\
1,&\text{ $s\ge 1$,}
\end{cases}
$$
where $s=\Delta^{P} _{i_{n_1}i_{n_2}...i_{n_k}...}$.
\end{lemma}

A method of  the corresponding proof is described in~\cite{Symon2017}.

Finally, let us describe integral properties.

\begin{theorem}
The Lebesgue integral of the function $G$ can be calculated by the
formula
$$
\int^1 _0 {G(s)ds}=\frac{\gamma_1p_1+\gamma_2p_2+\dots + \gamma_{q-1}p_{q-1}}{1-p_0r_0-p_1r_1 - \dots - p_{q-1}r_{q-1}}.
$$
\end{theorem}
\begin{proof}

Since $\gamma_0=0$,  let us denote the sum $\sum^{q-1} _{j=1}{\gamma_jp_j}$  by $a$ and let us denote the sum $\sum^{q-1} _{t=0}{p_tr_t}$ by $b$. Since equality~\eqref{eq: generalized shift 2} holds, we obtain
$$
s=p_{i_m}\sigma_m(s)+{(1-p_{i_m})}\Delta^P _{i_1i_2...i_{m-1}000...}+\beta_{i_m}\prod^{m-1} _{j=1}{p_{i_j}}
$$
and
$$
ds=p_{i_m}d(\sigma_m(s)).
$$

In the general case, for arbitrary positive integers $n_1$ and $n_2$, using equality \eqref{eq: 2-composition}, we have 
$$
p_{i_{n_2}}d\left(\sigma_{n_{2}}\circ \sigma_{n_1}(s)\right)=d\left( \sigma_{n_1}(s)\right).
$$
Also, for any posirive integer $k$, the following is true:
$$
d\left(\sigma_{n_{k-1}}\circ \dots \circ \sigma_{n_2}\circ \sigma_{n_1}(s)\right)=p_{i_{n_k}}d\left(\sigma_{n_k}\circ \sigma_{n_{k-1}}\circ \dots \circ \sigma_{n_2}\circ \sigma_{n_1}(s)\right),
$$
where $\sigma_0(s)=s$.

So, we have
\begin{align*}
\int^1 _0 {G(s)ds}&=\sum^{q-1} _{j=0}{\int^{\beta_{j+1}} _{\beta_j} {G(s)ds}}=\sum^{q-1} _{j=0}{\int^{\beta_{j+1}} _{\beta_{j}} {\left(\gamma_j+r_jG(\sigma_{\bar n_1}(s))\right)ds}}\\
&=\sum^{q-1} _{j=1}{\gamma_jp_j}+\sum^{q-1} _{t=0}{p_tr_t}\int^1 _0 {G(\sigma_{\bar n_1}(s))d(\sigma_{\bar n_1}(s))}\\
&=\sum^{q-1} _{j=1}{\gamma_jp_j}+\sum^{q-1} _{t=0}{p_tr_t}\left(\sum^{q-1} _{j=0}{\int^{\beta_{j+1}} _{\beta_{j}} {\left(\gamma_j+r_jG(\sigma_{\bar n_2}\circ \sigma_{\bar n_1}(s))\right)d(\sigma_{\bar n_1}(s))}}\right)\\
&=a+b\left(a+b\int^1 _0 {G(\sigma_{\bar n_2}\circ \sigma_{\bar n_1}(s)))d(\sigma_{\bar n_2}\circ \sigma_{\bar n_1}(s))}\right)\\
&=a+ab+b^2\left(\sum^{q-1} _{j=0}{\int^{\beta_{j+1}} _{\beta_{j+1}} {\left(\gamma_j+r_jG(\sigma_{\bar n_3}\circ\sigma_{\bar n_2}\circ \sigma_{\bar n_1}(s))\right)d(\sigma_{\bar n_2}\circ\sigma_{\bar n_1}(s))}}\right)\\
&=a+ab+b^2\left(a+b\int^1 _0 {G(\sigma_{\bar n_3}\circ\sigma_{\bar n_2}\circ \sigma_{\bar n_1}(s)))d(\sigma_{\bar n_3}\circ\sigma_{\bar n_2}\circ \sigma_{\bar n_1}(s))}\right)\\
&=a+ab+ab^2+b^3\left(a+b\int^1 _0 {g(\sigma_{\bar n_4}\circ\sigma_{\bar n_3}\circ\sigma_{\bar n_2}\circ \sigma_{\bar n_1}(s)))d(\sigma_{\bar n_4}\circ\sigma_{\bar n_3}\circ\sigma_{\bar n_2}\circ \sigma_{\bar n_1}(s))}\right)=\dots \\
\dots &= a+ab+\dots +ab^{k-1}\\
&+b^k\left(a+b\int^1 _0 {G(\sigma_{\bar n_{k+1}}\circ\sigma_{\bar n_k}\circ\ldots \circ \sigma_{\bar n_1}(s)))d(\sigma_{\bar n_{k+1}}\circ\sigma_{\bar n_k}\circ \ldots \circ \sigma_{\bar n_1}(s))}\right).
\end{align*}

Finally, we obtain
\begin{align*}
\int^1 _0{G(s)ds}&=\lim_{k\to\infty}{\left(\sum^{k} _{t=0}{ab^t}+b^{k+1}\int^1 _0 {G(\sigma_{\bar n_{k+1}}\circ\sigma_{\bar n_k}\circ\ldots \circ \sigma_{\bar n_1}(s)))d(\sigma_{\bar n_{k+1}}\circ\sigma_{\bar n_k}\circ \ldots \circ \sigma_{\bar n_1}(s))}\right)}\\
&=\sum^{\infty} _{k=1}{ab^{k-1}}=\frac{a}{1-b}=\frac{\gamma_1p_1+\gamma_2p_2+\dots + \gamma_{q-1}p_{q-1}}{1-p_0r_0-p_1r_1 - \dots - p_{q-1}r_{q-1}}.
\end{align*}
\end{proof}

{\section*{Statements and Declarations}}

\subsection*{Competing Interests}

\emph{The author states that there is no conflict of interest}

\subsection*{Information regarding sources of funding}

\emph{No funding was received}

\subsection*{Data availability statement}

\emph{The manuscript has no  associated data}

\subsection*{Study-specific approval by the appropriate ethics committee for research involving humans and/or animals, informed consent if the research involved human participants, and a statement on welfare of animals if the research involved animals (as appropriate)}

\emph{There are not suitable  for this research}

\end{document}